\documentclass[letterpaper]{article}

\title{Critical graphs without triangles: an optimum density construction}
\author{Wesley Pegden\footnote{
Department of Mathematics, 
Courant Institute, New York University,
251 Mercer St.,
New York, NY 10012. 
Email: pegden@math.nyu.edu}
}
\date{September 29, 2012}

\usepackage{amsmath,amsthm,amssymb,amscd,latexsym}
\usepackage{slashbox}
\usepackage{cite,pstricks,pst-node,multido}
\usepackage{url}

\newcommand{\sbs}{\subset}
\newcommand{\abs}[1]{\lvert #1 \rvert}

\newcommand{\cel}[1]{\left\lceil #1 \right\rceil}
\newcommand{\fcl}[1]{\left\lfloor #1 \right\rceil}
\newcommand{\flr}[1]{\left\lfloor #1 \right\rfloor}

\newcommand{\st}{\,\vrule\,}

\newcommand{\Z}{\mathbb{Z}}

\newtheorem{theorem}{Theorem}[section]

\newtheorem{question}[theorem]{Question}

\newtheorem{lemma}[theorem]{Lemma}
\newtheorem{cor}[theorem]{Corollary}

\newtheorem{observ}[theorem]{Observation}

{
\theoremstyle{definition}
\newtheorem{definition}[theorem]{Definition}

}

{
\theoremstyle{remark}

}

\newcommand{\defn}{\stackrel{\rm{defn}}{=}}
\newcommand{\ep}{\varepsilon}

\newcommand{\vx}[2]{\cnode*(#1){1.2pt}{#2}}
\newcommand{\VX}[2]{\cnode*(#1){2.5pt}{#2}}

\newcommand{\edg}[2]{\ncline{#1}{#2}}

\newcommand{\cedg}[2]{\ncarc[arcangleA=15,arcangleB=15]{#1}{#2}}

\begin{document}

\maketitle

\begin{abstract}
We construct dense, triangle-free, chromatic-critical graphs of chromatic number $k$ for all $k\geq 4$.  For $k\geq 6$ our constructions have $>(\frac 1 4 -\ep)n^2$ edges, which is asymptotically best possible by Tur\'an's theorem.  We also demonstrate (nonconstructively) the existence of dense $k$-critical graphs avoiding all odd cycles of length $\leq \ell$ for any $\ell$ and any $k\geq 4$, again with a best possible density of $> (\frac 1 4 -\ep)n^2$ edges for $k\geq 6$.  The families of graphs without triangles or of given odd-girth are thus rare examples where we know the correct maximal density of $k$-critical members ($k\geq 6$).
\end{abstract}

\section{Introduction}
\label{intro}

The family of graphs without triangles exhibits many similarities with that of the bipartite graphs---for example, in the number of such graphs on a given number of vertices\cite{ekr}---suggesting that in some senses, the most significant barrier to being bipartite is the presence of triangles.  This naturally motivates the study of the chromatic number of triangle-free graphs, which, though not apparent at first sight, can be arbitrarily large, due to the constructions (for example) of Zykov and Mycielski\cite{zykov},\cite{mycielski}.  Both of these constructions have in common that they give graphs which are quite sparse.  Mycielski's is the denser of the two constructions, with $O(n^{\log_2 3})$ edges, so well below quadratic edge density.  Coupled with the observation that a triangle-free graph with the maximum number of edges is a balanced complete bipartite graph (and so 2-colorable), this suggests the chromatic number of triangle-free graphs with lots of edges as a natural area of study.

If we are not careful we encounter trivialities.  Taking the disjoint union of a $k$-chromatic triangle-free graph with a large complete bipartite graph, we get a rather disappointing example of a $k$-chromatic triangle-free graph which is dense (in this case, it can have $\frac {1-\ep}{4}n^2$ edges).  Since the added complete bipartite graph plays no role in either the chromatic number of the graph or the difficulty of avoiding triangles in the `important' part of the graph, this is hardly satisfactory.

The question along these lines by Erd\H{o}s and Simonovits was: how large can the minimum degree be in a triangle-free graph of large chromatic number?  Hajnal's construction given in \cite{mindeg} proceeds by pasting new vertices and edges onto a suitable Kneser graph (triangle-free and already with arbitrarily large chromatic number) and shows that one can have triangle-free graphs with minimum degree $\delta \geq (1-\ep)\frac n 3$ and arbitrarily large chromatic number.  They originally conjectured that $\delta >\frac n 3$ would force 3-colorability.  H\"aggkvist \cite{h} found a 4-chromatic counterexample.  This problem has recently been resolved: after Thomassen proved that triangle-free graphs with minimum degree $\delta>\frac {1+\ep} 3 n$ have bounded chromatic number\cite{t}, Brandt and Thomass\'e have shown that triangle-free graphs with $\delta>\frac 1 3 n$ are indeed 4-colorable\cite{bt}.  For 3-colorability, the threshold is $\frac{10}{29}$\cite{jin}. Remaining questions appear in \cite{bt}.  For more about triangle-free graphs with large minimum degree, we refer to \cite{b}.

\vspace{1em}
\noindent 
The problem we study is of a slightly different type.  We return to the basic notion of density (requiring only a quadratic number of edges, with no restrictions on the minimum degree), and use the requirement of \emph{criticality} to avoid trivialities like adding a disjoint complete bipartite graph.  
\begin{definition}
A graph is $k$-critical if it has chromatic number $k$, and removing any edge allows it to be properly $(k-1)$-colored.  
\end{definition}
\noindent Thus by requiring this of our triangle-free graphs, there will be no edges contributing to the density of the graph without playing a role in its chromatic number.  Many of the classical constructions of $k$-chromatic triangle-free graphs in fact give critical graphs: the Mycielski construction gives $k$-critical graphs, and the Zykov construction can be easily modified to do the same.  Lower bounds on the sparsity of triangle-free critical graphs have been given by \cite{ks}.

\begin{figure}[t]
\begin{center}
\psset{xunit=.9cm,yunit=.7cm}
\begin{pspicture}(0,1.5)(6,5)

\psset{linewidth=.5pt}

\multido{\n=1+1}{5}{%
  \VX{.5,\n}{Vfa\n}
  \VX{5.5,\n}{Vfd\n}
  \VX{1.5,\n}{Vfb\n}
  \VX{4.5,\n}{Vfc\n}
  \edg{Vfa\n}{Vfb\n}
  \edg{Vfc\n}{Vfd\n}
}

\multido{\na=1+1,\nb=2+1}{4}{%
  \edg{Vfa\na}{Vfa\nb}
  \edg{Vfd\na}{Vfd\nb}
}

\multido{\ni=1+1}{5}{%
  \multido{\nj=1+1}{5}{%
    \edg{Vfb\ni}{Vfc\nj}
  }
}

\ncarc[arcangleA=30,arcangleB=30]{Vfa1}{Vfa5}
\ncarc[arcangleA=30,arcangleB=30]{Vfd5}{Vfd1}

\end{pspicture}
\end{center}
\caption{Toft's graph is a dense, triangle-free, 4-critical graph.  It is constructed by joining 2 sets of size $2\ell+1$ in a complete bipartite graph, and matching each to an (odd) $(2\ell+1)$-cycle.  \label{t} It has $\frac 1 {16}n^2+n$ edges.}
\end{figure}
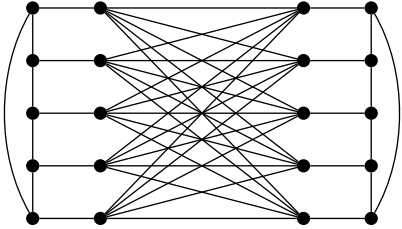

\begin{figure}[t]
\begin{center}
\begin{pspicture}(1,1.3)(5,3)

\psset{linewidth=.5pt}

\VX{1,1}{g1}
\rput(.9,1.2){$v$}
\VX{2,1}{g2}
\VX{3,1}{g3}
\VX{4,1}{g4}
\VX{5,1}{g5}

\edg{g1}{g2}
\edg{g2}{g3}
\edg{g3}{g4}
\edg{g4}{g5}
\cedg{g5}{g1}

\VX{1,2}{m1}
\rput(.9,2.2){$\bar v$}
\VX{2,2}{m2}
\VX{3,2}{m3}
\VX{4,2}{m4}
\VX{5,2}{m5}

\edg{m1}{g2} \edg{m1}{g5}
\edg{m2}{g1} \edg{m2}{g3}
\edg{m3}{g2} \edg{m3}{g4}
\edg{m4}{g3} \edg{m4}{g5}
\edg{m5}{g4} \edg{m5}{g1}

\VX{3,3}{x}
\rput(3.3,3.1){$x$}

\edg{m1}{x}
\edg{m2}{x}
\edg{m3}{x}
\edg{m4}{x}
\edg{m5}{x}

 %




\end{pspicture}
\end{center}
\caption{Mycielski's construction of $k$-critical triangle-free graphs.  Shown is $\mu(C_5)$.\label{mycielski}}
\end{figure}
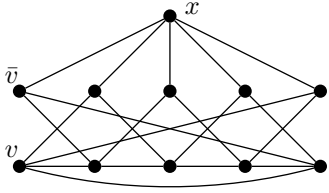

The main question we answer in the affirmative here (Theorem \ref{dtf}) is: \textbf{does there exist a family of triangle-free critical graphs of arbitrarily large chromatic number, and each with $>cn^2$ edges for some fixed $c>0$?}  

\vspace{1em}

\begin{definition}
The \emph{density constants} $c_{\ell,k}$ are each defined as the supremum of constants $c$ such that there are families of (infinitely many) $k$-critical graphs of odd-girth $> \ell$ and with $>cn^2$ edges.
\end{definition}
\noindent (The \emph{odd-girth} is the length of a shortest odd-cycle.) We abbreviate $c_{3,k}$ as $c_k$.  Thus the positive answer to the above question will follow from our result that in fact $c_k$ is bounded below by a constant for all $k\geq 4$ (Theorem \ref{dtf}).

An easier question is to ask just for dense families of (infinitely many) triangle-free $k$-critical graphs for \emph{each} $k\geq 4$---so, just to show that all of the constants $c_k$, $k\geq 4$ are positive.  For $k=4$, this is demonstrated by Toft's graph (Figure \ref{t}) \cite{toft}, which has $>\frac {n^2} {16}$ edges (the constant $\frac{1}{16}$ is the best known for $k=4$ even without the condition `triangle-free').  Using Toft's graph, one can repeatedly apply Mycielski's operation \cite{mycielski} to get families of $k$-critical triangle-free graphs with $\geq c_kn^2$ edges for each $k\geq 4$.  Here $c_k=\frac {1}{16}(\frac{3}{4})^{k-4}$, which tends to 0 with $k$---thus this method only answers the easier question.  (Mycielski's operation $\mu(G)$ on $G$ is: create a new vertex $\bar v$ for each $v\in G$, join it to the neighbors of $v$ in $G$, and join a new vertex $x$ to all of the new vertices $\bar v$.  This operation is illustrated in Figure \ref{mycielski} where $G$ is the 5-cycle.)

The weaker question is also answered directly for the particular case $k=5$ by the following construction of Gy\'arf\'as\cite{g}.  Four copies $T_1,\dots,T_4$ of the Toft graph are matched each with corresponding independent sets $A_1,\dots,A_4$.  The pairs $A_i,A_{i+1}$, $1\leq i\leq 3$, are each joined in complete bipartite graphs, and the sets $A_1,A_4$ are both joined completely to another vertex $x$.  This is illustrated in Figure \ref{gyarfas}.  This gives a constant $c_5\geq \frac{13}{256}$.  Our construction will show that $c_5\geq \frac 4 {31}$, which is asymptotically the same density shown by Toft \cite{toft} without the triangle-free requirement, and still the best known for that problem.  Our general result for the density of $k$-critical triangle-free graphs is the following, which we prove constructively.

\begin{figure}[t]
\begin{center}
\psset{xunit=.36cm,yunit=.22cm}
\begin{pspicture}(0,5)(29,19)
%
%
\psset{linecolor=black,linewidth=2pt}
\psline(5.5,10.5)(7.5,12.5)
\psline(7.5,10.5)(5.5,12.5)
\psline(5.5,9.5)(7.5,11.5)
\psline(7.5,9.5)(5.5,11.5)
\psline(13.5,10.5)(15.5,12.5)
\psline(15.5,10.5)(13.5,12.5)
\psline(13.5,9.5)(15.5,11.5)
\psline(15.5,9.5)(13.5,11.5)
\psline(21.5,10.5)(23.5,12.5)
\psline(23.5,10.5)(21.5,12.5)
\psline(21.5,9.5)(23.5,11.5)
\psline(23.5,9.5)(21.5,11.5)

%
%

\Multido{\nheight=15+0, \nheightt=16+0, \nheighte=19+0, \nmheight=9+0, \nnn=1+8, \nax=1+8, \nbx=2+8, \ncx=3+8, \ndx=4+8}{4}{%
\psset{linecolor=black,linewidth=.3pt}
  \multido{\ngf=\nheight+1,\nggf=\nmheight+1,\nmoheight=\nheight+1}{5}{%
    \vx{\nax,\ngf}{V\nax B\ngf}
    \vx{\nbx,\ngf}{V\nbx B\ngf}
    \vx{\ncx,\ngf}{V\ncx B\ngf}
    \vx{\ndx,\ngf}{V\ndx B\ngf}
    \vx{\nax,\nggf}{V\nax B\nggf}
    \vx{\nbx,\nggf}{V\nbx B\nggf}
    \vx{\ncx,\nggf}{V\ncx B\nggf}
    \vx{\ndx,\nggf}{V\ndx B\nggf}
    \psset{linecolor=black,linewidth=.25pt,dash=3pt 1.5pt}
    \ncarc[arcangleA=60,arcangleB=60]{V\nax B\nggf}{V\nax B\nmoheight}
    \ncarc[arcangleA=25,arcangleB=25]{V\nbx B\nggf}{V\nbx B\nmoheight}
    \ncarc[arcangleA=25,arcangleB=25]{V\ncx B\nmoheight}{V\ncx B\nggf}
    \ncarc[arcangleA=60,arcangleB=60]{V\ndx B\nmoheight}{V\ndx B\nggf}
%
    \edg{V\nax B\ngf}{V\nbx B\ngf}
    \edg{V\ncx B\ngf}{V\ndx B\ngf}
  }
  \multido{\ngfa=\nheight+1,\ngfb=\nheightt+1}{4}{%
    \edg{V\nax B\ngfa}{V\nax B\ngfb}
    \edg{V\ndx B\ngfa}{V\ndx B\ngfb}
  }
  \multido{\ngfi=\nheight+1}{5}{%
    \multido{\ngfj=\nheight+1}{5}{%
      \edg{V\nbx B\ngfi}{V\ncx B\ngfj}
    }
  }
  \ncarc[arcangleA=30,arcangleB=30]{V\nax B\nheight}{V\nax B\nheighte}
  \ncarc[arcangleA=30,arcangleB=30]{V\ndx B\nheighte}{V\ndx B\nheight}
%
}

\psset{linecolor=black}
\vx{15,2}{X}
\psset{linecolor=black,linewidth=.005}
\multido{\nstartB=1+1}{4}{%
  \multido{\nhanyB=9+1}{5}{%
    \edg{V\nstartB B\nhanyB}{X}
  }
}
\multido{\nstartB=25+1}{4}{%
  \multido{\nhanyB=9+1}{5}{%
    \edg{V\nstartB B\nhanyB}{X}
  }
}

\end{pspicture}
\end{center}
\caption{Gy\'arf\'as's graph: a dense, triangle-free, 5-critical graph.  Each of the three double crosses stands for the edges of the complete bipartite graph joining the independent sets on either side. (For this particular $n=161$, each double cross represents 400 edges).   \label{gyarfas} Gy\'arf\'as's graph is dense, with $(\frac{13}{256}+o(1))n^2$ edges.}
\end{figure}
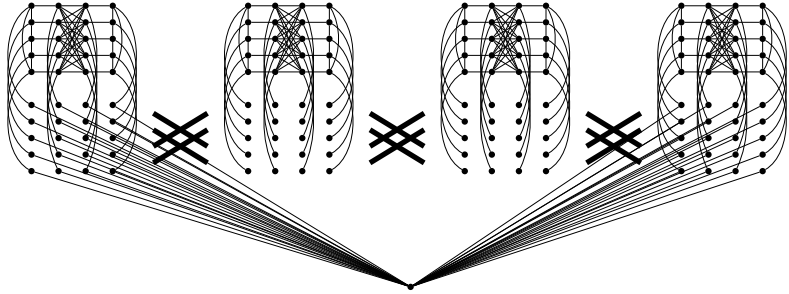

\begin{theorem}
There is a dense family of triangle-free critical graphs containing infinitely many $k$-chromatic graphs for each $k\geq 4$, obtained by recursive construction.  In particular, the density constants satisfy $c_4\geq \frac 1 {16}$, $c_5\geq \frac{4}{31}$, and $c_k=\frac 1 4$ for all $k\geq 6$.
\label{dtf}
\end{theorem}
\noindent 
In the case $k=4$ our construction is the same as Toft's graph.  For $k\geq 6$, the upper bound of $\frac 1 4$ follows from Tur\'an's theorem just from the fact that the graphs don't have triangles.

Theorem \ref{dtf} follows from our constructions in Section \ref{construct}, which can be seen as generalizing the Toft graph.  The basic idea of the construction is the following: the Toft and Gy\'arf\'as constructions work by pasting graphs onto independent sets so that in any $(k-1)$-coloring of the graph, at least two different colors occur in each of the independent sets (and the `pasting' has suitable criticality-like properties as well).  In Section \ref{construct}, we recursively define a way of `critically' pasting graphs to independent sets so that we can be assured that in a $(k-1)$-coloring of a graph, $\fcl{\frac k 2}$ colors occur in each independent set (i.e., $\flr{\frac k 2}$ in one and $\cel{\frac k 2}$ in the other).  Moreover the pasted graphs will be negligible in size compared with the independent sets.  From this point, we simply defer back to the idea of Toft's construction, joining two such independent sets in a complete bipartite graph.  (One could also use Gy\'arf\'as's graph as model for the general case, but this gives less dense graphs.)

A natural extension of the triangle-free case is of course to avoid higher-order odd-cycles.  A theorem of Stiebitz on iterated generalized Mycielski operations, in combination with our construction, will allow us to address this case as well.
\begin{theorem}
\label{t.og}
For each $k\geq 4$, $\ell\geq 5$, there is a dense family of critical graphs of odd-girth $>\ell$ whose members have unbounded chromatic number.  In particular, the density constants satisfy $c_{\ell,4}\geq \frac{1}{(\ell+1)^2}$, $c_{\ell,5}\geq \frac{1}{2(\ell+1)}$, and $c_{\ell,k}=\frac 1 4$ for $k\geq 6$.
\end{theorem}
\noindent Theorem \ref{t.og} is nonconstructive, since one step of the `construction' involves a deletion argument; the proof is in Section \ref{nopents}.  The density for the cases $k\geq 6$ is best possible even if we only wanted to avoid odd cycles of length exactly $\ell$ as a consequence of the classical Erd\H{o}s-Stone theorem \cite{es}, since odd-cycles are 3-chromatic.

Theorems \ref{dtf} and \ref{t.og} determine the exact values of $c_{\ell,k}$ for $k\geq 6$ and so the triangle-free graphs, and more generally, odd-girth-$\ell$ graphs, are perhaps the only examples of natural families of graphs where the maximum possible edge-density in $k$-critical graphs is now asymptotically known ($k\geq 6$); moreover, in this case, the requirement of $k$-criticality does not affect the maximum possible edge density asymptotically. In terms of the constants $c_{\ell,k}$, the results of Theorems \ref{dtf} and \ref{t.og} are summarized in Table \ref{ctable} (in Section \ref{Qs}).

\subsection*{Acknowledgment}
I'd like to thank Andr\'as Gy\'arf\'as for some helpful discussions on this question; particularly, for giving his construction (Figure \ref{gyarfas}) for the case $k=5$.  I am also indebted to an anonymous referee for mentioning the possible relevance of Theorem \ref{t.stieb}; this allowed extending the nonconstructive Theorem \ref{t.og} from the case of odd-girth 7 to arbitrary odd-girth.

\section{Avoiding triangles (constructions)}
\label{construct}
To construct critical triangle-free graphs, we will construct graphs $U$ which have large independent sets (essentially as large as $U$ for $k\geq 6$ in fact) in which $\fcl{\frac k 2}$ colors must show up in a $(k-1)$-coloring.  (In fact, our techniques would allow us to require that all $k-1$ colors show up in the independent set, but this is not the best choice from a density perspective.)  The $U$'s will have criticality-like properties as well (Lemma \ref{setcolor}), and  we will be able to join the relevant independent sets from two such $U$'s in a complete bipartite graph to get our construction.  Let us now specify the construction of the sets $U$.

Given graphs $S_1,S_2,\dots S_t$, we construct $U(S_1,S_2,\dots,S_t)$ as follows: take the (disjoint) union of the graphs $S_i$, together with the independent set $A=\prod V(S_i)$, which we call the \emph{active} set of vertices, and join vertices $v\in S_i$ to vertices $u\in A$ which equal $v$ in their $i$'th coordinate.  This construction is illustrated for $t=2$ in Figure \ref{Us}.  Where $r_1,r_2,\dots r_t$ are positive integers, we write $U(r_1,r_2,\dots r_t)$ to mean a construction $U(S_1,S_2,\dots,S_t)$ where, for each $i$,  $S_i$ is a triangle-free $r_i$-critical graph without isolated vertices.

With just one graph $S$, $U(S)$ simply consists of a matching between $S$ and an independent set of the same size. Thus the Toft graph is two identical copies of a $U(3)$, whose active sets are joined in a complete bipartite graph.  Gy\'arf\'as's construction is now four copies $C_1,C_2,C_3,C_4$ of a $U(4)$, where the active sets of $C_i$ and $C_{i+1}$ are joined in a complete bipartite graph for $1\leq i\leq 3$, and the active sets of both $C_1$ and $C_4$ are completely joined to a new vertex $x$.

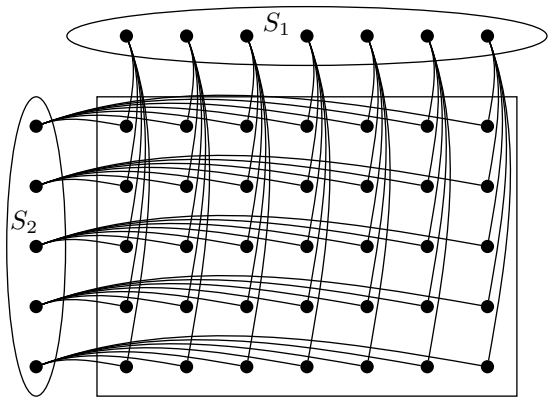
\begin{figure}[t!]
  \begin{center}
    \psset{xunit=.8cm,yunit=.8cm,linewidth=.5pt}
    \begin{pspicture}(0,1)(9,7)
      \psellipse(5,6.5)(4,.5)
      \rput(4.5,6.7){$S_1$}
      \VX{2,6.5}{s11}
      \VX{3,6.5}{s12}
      \VX{4,6.5}{s13}
      \VX{5,6.5}{s14}
      \VX{6,6.5}{s15}
      \VX{7,6.5}{s16}
      \VX{8,6.5}{s17}

      \psellipse(.5,3)(.5,2.5)
      \rput(0.3,3.4){$S_2$}
      \VX{.5,1}{s21}
      \VX{.5,2}{s22}
      \VX{.5,3}{s23}
      \VX{.5,4}{s24}
      \VX{.5,5}{s25}

\psframe(1.5,.5)(8.5,5.5)

      \multido{\nax=2+1,\ncx=1+1}{7}{
	\multido{\nay=1+1,\ncy=1+1}{5}{
	  \VX{\nax,\nay}{a\ncx\ncy}
	  \ncarc[arcangleA=20,arcangleB=10]{s1\ncx}{a\ncx\ncy} 
	  \ncarc[arcangleA=20,arcangleB=10]{s2\ncy}{a\ncx\ncy} 
	}
      }


    \end{pspicture}
  \end{center}
\caption{Constructing $U(S_1,S_2)$ from the graphs $S_1$ and $S_2$.  The rectangle encloses the set of active vertices.\label{Us}  (Edges within the $S_i$ are not drawn.)}
\end{figure}

\begin{observ}
$U(r_1,r_2,\dots,r_t)$ is triangle-free, and if $s_i=\abs {S_i}$, it has $s_1\cdots s_t$ active vertices, and $s_i$ structural vertices of type $i$.
\qed
\label{Utfs}
\end{observ}
\noindent The \emph{structural vertices} of type $i$ are vertices from the copies of $S_i$ used in the construction of $U(r_1,r_2,\dots r_t)$.

Our inductive proof of Lemma \ref{setcolor} will use the following properties of the sets $U(S_1,\dots,S_t)$:
\begin{observ}
  Let $A_{\omega}\sbs \prod V(S_i)$ be the set of all vertices of the active set whose $t$\emph{th} coordinate is the vertex $\omega\in S_t$.  Then the subgraph of $U(S_1,\dots S_t)$ induced by the set $A_\omega\cup \bigcup\limits_{1\leq i< t} V(S_i)$ is isomorphic to $U(S_1,\dots,S_{t-1}).$\label{induct}\qed
\end{observ}
\begin{observ}
 \label{extend} Any colorings of the graphs $S_1,S_2,\dots S_t$ can be extended to the rest of $U(S_1,S_2,\dots,S_t)$ (\emph{i.e.}, to the active vertices) so that just $t+1$ different colors of our choosing appear in the active set.
\end{observ}
\begin{proof}
This follows from observing that the degree of each active vertex in $U(S_1,S_2,\dots, S_t)$ is $t$.
\end{proof}

Before proceeding any further, we point out the following basic fact, which lies at the heart of all recursive constructions of critical graphs:
\begin{observ}
A $k$-critical graph without isolated vertices can be $k$-colored so that color $k$ occurs only at a single vertex of our choice.
\label{vc}
\end{observ}
\begin{proof}
This follows from observing that such a graph is \emph{vertex-critical}; i.e., removing any vertex decreases the chromatic number.
\end{proof}
\noindent We are now ready to prove the main lemma, about the sets  
\[
U^{k-1}_{k-i+1}\defn U(k-1,k-2,\dots,k-i+1).
\]  
Basically, it says that in a $(k-1)$-coloring, their active sets have a set of properties similar to criticality.

\begin{lemma} The sets $U^{k-1}_{k-i+1}= U(k-1,k-2,\dots,k-i+1)$ satisfy:
\begin{enumerate}
\item  \label{scolor} $U^{k-1}_{k-i+1}$ can be properly $(k-1)$-colored so that  $\leq i$ different colors appear at active vertices and so that among active vertices, only one vertex of our choosing gets the $i$th color.
\item  In any $(k-1)$-coloring of  $U^{k-1}_{k-i+1}$, at least $i$ different colors occur as the colors of active vertices.
\label{ascolor}
\item If any edge from $U_{k-i+1}^{k-1}$ is removed, it can be properly $(k-1)$-colored so that at most $i-1$ colors occur at active vertices.
\label{ascritical}
\end{enumerate}

\label{setcolor}
\end{lemma}
\begin{proof}
We prove Lemma \ref{setcolor} by induction on $i$.  For $i=1$, we interpret $U^{k-1}_k$ as simply a single active vertex and the statement is trivial. Recall that the graphs $S_j$ ($1\leq j\leq i-1$) are the triangle-free $(k-j)$-critical graphs used in the construction of $U^{k-1}_{k-i+1}$.  For $i>1$, Observation \ref{induct} tells us that $U^{k-1}_{k-i+1}$ consists of $S_{i-1}$ together with $s_{i-1}=\abs{S_{i-1}}$ copies $C_\omega$ of a $U^{k-1}_{k-i+2}$, one for each vertex $\omega\in S_{i-1}$, which are pairwise disjoint except at the graphs $S_1,S_2,\dots, S_{i-2}$ (in other words, their active sets $A(C_\omega)$ are disjoint).  Note that each $\omega\in S_{i-1}$ is adjacent to every vertex in $A(C_\omega)$.  

To prove part \ref{scolor}, color $S_{i-1}$ with colors $i-1$ through $k-1$ so that color $i-1$ occurs only at vertex $\omega_0$ (we can do this by Observation \ref{vc}).  By induction, $(k-1)$-color $C_{\omega_0}$ so that its active vertices get colors from $1,2,\dots,i-2,i$ and so that color $i$ occurs at only one vertex $u$ of $C_{\omega_0}$.  Now we can use Observation \ref{extend} to extend the current partial $(k-1)$-coloring to the rest of the $C_\omega$'s so that their active vertices get colors from 1 through $i-1$; this gives us a $(k-1)$-coloring of $U^{k-1}_{k-i+1}$ where active vertices get colors from 1 through $i$, and $i$ occurs only at $u$.  Finally, note that we could have chosen so that $u$ was any active vertex.

For part \ref{ascolor}, notice that by induction, $i-1$ different colors occur at active vertices of the $C_\omega$'s in a $(k-1)$-coloring of $U^{k-1}_{k-i+1}$.  Thus if only $i-1$ colors appear at active vertices of $U^{k-1}_{k-i+1}$, the same set of $i-1$ colors occurs at the sets of active vertices of each of the $C_\omega$'s; but then $S_{i}$ is colored with a disjoint set of colors, but then we need $(i-1)+(k-i+1)>k-1$ colors overall, a contradiction.

For part \ref{ascritical}, assume first that the removed edge came from $S_{i}$:  Then we can $(k-i)$-color what remains, and (inductively by part \ref{scolor}) color the $C_\omega$'s so that only the leftover $i-1$ colors occur at the active vertices.  

If the removed edge was one joining a vertex $\omega_0$ in $S_i$ to a vertex $v$ in the active set of $C_{\omega_0}$, we color $S_i$ with the colors $i-1$ through $k-1$ so that color $i-1$ occurs only at vertex $\omega_0$ (by Observation \ref{vc}), and color $C_{\omega_0}$ so that only colors $1,2,\dots,i-1$ occur at its active vertices, and so that color $i-1$ occurs only at $v$ among active vertices (by part \ref{scolor}, inductively).  Coloring the other $C_\omega$'s so that only colors $1,2,\dots,i-1$ occur at their active vertices (again part \ref{scolor}), we have a proper $(k-1)$-coloring of $U^{k-1}_{k-i+1}$ where only colors 1 through $i-1$ appear at active vertices.

If instead the removed edge was from one of the copies $C_{\omega_0}$, we color that copy (by induction) so that only colors $1,2,\dots,i-2$ occur at its active vertices, and $S_{i}$ with colors $i-1$ through $k-1$ so that $i-1$ occurs only at the vertex $\omega_0$ (by Observation \ref{vc}).  Then coloring the rest of the $C_{\omega}$ so that only colors $1$ through $i-1$ appear at active vertices (by part \ref{scolor}), we have a proper $(k-1)$-coloring of $U^{k-1}_{k-i+1}$ where only colors 1 through $i-1$ appear at active vertices.
\end{proof}

\subsection{Constructing $G_k$}
\label{gk}
We take $G_k$ to be the union of a $C_1=U^{k-1}_{\cel{\frac k 2} +1}$ with a $C_2=U^{k-1}_{\flr{\frac k 2}+1}$; the active sets are joined in a complete bipartite graph.  (For $k=4$ this gives the Toft graph.)  $G_k$ is $k$-colorable, since part \ref{scolor} of Lemma \ref{setcolor} implies that we can $(k-1)$-color each $C_i$ so that the sets of active vertices of each get just $\flr{\frac k 2}$ and $\cel{\frac k 2}$ colors, respectively.  Moreover part \ref{ascolor} of Lemma \ref{setcolor} implies that $G_k$ cannot be $(k-1)$-colored.  $G_k$ is triangle-free since the $C_i$'s are triangle-free, the new edges form a bipartite graph, and they are added to an independent set of vertices.

For criticality there are now just two cases to check.  If a removed edge belongs to (for example) $C_1$, then by part \ref{ascritical} of Lemma \ref{setcolor} we can color $C_1$ with colors $1,\dots,k-1$ so that only colors $1,\dots,\flr{\frac k 2}-1$ occur at its active vertices.  Then, coloring $C_2$ with $1,\dots,k-1$ so that only $\flr{\frac k 2},\flr{\frac k 2}+1,\dots, k-1$ occur at its active vertices, we have given a proper $(k-1)$-coloring.

Otherwise the removed edge is between active vertices of $C_1$, and $C_2$.  We can color each of the $C_i$ with colors $1,\dots,k-1$ so that only colors $1,\dots ,\flr{\frac k 2}$ appear at active vertices of $C_1$, and only $\flr{\frac {k} 2},\dots, k-1$ appear at active vertices of $C_2$, and (by part \ref{scolor} of Lemma \ref{setcolor}) require that color $\flr{\frac k 2}$ occurs only at the end-vertices of the removed edge, so that this $(k-1)$-coloring is proper.

 \subsection{Density of $G_k$ (for $k\geq 6$)}
\label{density}
 To estimate the density of $G_k$, recall that  $C_p$, $p\in \{1,2\}$, will have $s^p_1s^p_2\cdots s^p_{t_p}$  active vertices and $s^p_j$ structural vertices of type $j$, where here $s^p_j$ is the number of vertices in the graph $S^p_j$, the $(k-j)$-critical graph used in the construction of $C_p=U^{k-1}_{\fcl{\frac k 2}+1}=U(k-1,k-2,\dots,\fcl{\frac k 2}+1)$.  

Now so long as $k\geq 6$, we have that $t_1,t_2\geq 2$ and so by choosing $s^p_1,s^p_2$ suitably large for each $p\in \{1,2\}$, we assure that a $1-\ep'$ fraction of the vertices from each $C_p$ are active, giving that $G_k$ has $\geq (\frac 1 4-\ep)n^2$ edges for all $k\geq 6$, so long as we can take $C_1,C_2$ to be the same size.

For the case where $k$ is even, we can let them be identical, so if we are not interested in good constants for the case where $k$ is odd, we could finish by using the Mycielski operation to get the case `$k$ is odd' from the case `$k$ is even' (giving a density constant of $c_k\geq \frac{3}{16}$ for the odd $k$).  To prove best-possible density constants, we will want to argue that we can take $C_1,C_2$ to be the same size even if $k$ is odd.  This will follow from Lemma \ref{homap} below.

An arithmetic progression is called homogeneous if it includes the point 0.  We call a subset of $\Z^+$ \emph{positive homogeneous} if it contains (all) the positive terms of an infinite homogeneous arithmetic progression.  We have:
\begin{observ}
  The sums and products of positive homogeneous sets are positive homogeneous sets.\qed
\label{+*hp}
\end{observ}
\noindent The important property we need is this:
\begin{observ}
The intersection of two positive homogeneous sets is a positive homogeneous set.\qed
\label{caphp}
\end{observ}
\noindent In particular, the intersection is infinite, and so nonempty.  Because of this property, the following lemma regarding the $U^{k-1}_{k-i+1}$ implies that we can let $\abs{C_1}=\abs{C_2}$ in the case where $k$ is odd, finishing the proof of the density of the $G_k$:
\begin{lemma}
Let $k\geq 4,$ $1\leq i\leq k$, and fix triangle-free $(k-j)$-critical graphs $S_j$ of order $s_j$ for each $j>2$.  Then the set of natural numbers $n$ for which there is a $U^{k-1}_{k-i+1}$ on $n$ points constructed using the fixed graphs $S_2,S_3,\dots$ (so, we only have freedom in choosing the graph $S_1$) is positive homogeneous.  The same holds for the set of natural numbers $n'$ for which there exists a $U^{k-1}_{k-i+1}$ constructed from the $S_2,S_3,\dots$ which has exactly $n'$ active vertices.
\label{homap}
\end{lemma}
\begin{proof}
This follows by induction from Observations \ref{Utfs}, \ref{+*hp}, and \ref{caphp}.
\end{proof}
\noindent This completes the proof that the $G_k$ are dense with $(\frac{1}{4}-o(1))n^2$ edges for $k\geq 6$.  For $k=4$ our construction is the same as Toft's graph and so has $\frac 1 {16} n^2+n$ edges.  The remaining case is $k=5$.

\subsection{Density of $G_5$}
\label{g5}
In this section we show the bound $c_5\geq \frac{4}{31}$.

In the construction of $G_5$, $C_1$ is a $U(4)$, and $C_2$ is a $U(4,3)$.  Say $C_1$ consists of a copy $S_1^1$ of the Toft graph  matched with an independent set of the same size, and $C_2$ is likewise constructed from the $4$- and $3$-critical graphs $S^2_1,S^2_2$, respectively.  Denoting by $v(H)$ and $e(H)$ the number of vertices and edges, respectively, in a graph $H$, we have:
\[
e(G_5)>e(S_1^1)+v(S_1^1)\cdot \abs{A(C_2)}.
\]
(Of course, here $\abs{A(C_2)}=v(S^2_1)\cdot v(S^2_2)$.)  Since $S_1^1$ is the Toft graph, we have
\begin{equation}
  e(G_5)>\frac 1 {16} v(S_1^1)^2+v(S_1^1)\cdot \abs{A(C_2)}.
  \label{g5e}
\end{equation}
For vertices instead of edges, we have
\begin{equation*}
  v(G_5)=2 v(S_1^1)+v(S^2_1)\cdot v(S^2_2)+v(S^2_1)+v(S^2_2).
\label{g5v}
\end{equation*}
When we take $v(S^2_1)$, $v(S^2_2)$ to be large, the last two summands  are negligible:
\begin{equation}
  v(G_5)<(1+\ep)\left(2 v(S_1^1)+\abs{A(C_2)}\right).
  \label{ag5v}
\end{equation}
We want to maximize the ratio $\frac{e(G_5)}{v(G_5)^2}$.  It is an easy optimization problem to check that our best choice is to make $\abs{A(C_2)}$ larger than $v(S_1)$ by a factor of $\frac {15}{8}$ (note that we can do this by Lemma \ref{homap} and Observation \ref{caphp}).  For this choice, and denoting now $v(S^1_1)$ by simply $s^1_1$, lines (\ref{g5e}) and  (\ref{ag5v}) give:
\begin{equation}
  \frac{e(G_5)}{v(G_5)^2}>\frac 1 {1+\ep}
\frac
{
  \frac 1 {16} (s^1_1)^2+\frac {15}{8}(s_1^1)^2
}
{
  (\frac{31}{8})^2(s^1_1)^2
}
=\left(\frac{1}{1+\ep}\right) \frac 4 {31},
\end{equation}
and thus, taking the supremum (as $n$ grows large), we do have that $c_5\geq \frac 4 {31}$.

\section{Avoiding more short odd cycles}
\label{nopents}

By using an iterated generalized Mycielski operation together with the constructive method used to prove Theorem \ref{dtf} in Section \ref{construct}, we will demonstrate the existence of families of dense $k$-critical graphs for $k\geq 4$ of any fixed odd-girth $\ell+2$.  In particular, we show the best-possible bounds $c_{\ell,k}\geq \frac 1 4$ for all $k\geq 6$.

We first define the generalization of the Mycielskian doubling operation: to construct $\mu_q(M)$, add for each vertex $v=v_0$ in the graph $M$ vertices $v_1,v_2,\dots,v_q$ and join $v_i$ to all vertices $w_{i-1}$ (and $w_{i+1}$) for which $v,w$ are adjacent in $M$. We call the vertices $\{v_q\st v\in M\}$ the \emph{forward} vertices of $\mu_q(M)$.  We call all edges of the original graph $M$ Type 1 edges, edges between between the sets $\{v_{q-1}\}$ and $\{v_q\}$ Type 3 edges; and all other edges Type 2 edges.    We define $\bar \mu_q(M)$ to be $\mu_q(M)$ plus a single vertex $z$ joined to all the forward vertices $\{v_q\st v\in M\}$.  $\bar \mu_q(M)$ is sometimes called a \emph{cone} of the graph $M$ (for example, in \cite{cones}).  Note that $\bar \mu_1(M)$ is simply the Mycielski operation applied to $M$.

Letting $\bar \mu_q^0(M):=M$, we now recursively set $\bar \mu_q^k(M)=\bar \mu_q\left(\bar \mu_q^{k-1}(M)\right)$ and $\mu_q^k(M)=\mu_q\left(\bar \mu_q^{k-1}(M)\right)$; the two sides of this last equation have the same sets of forward vertices, Type 1 edges, Type 2 edges, and Type 3 edges.  Note that these sets are not recursively inherited---for example, \emph{all} edges of $\mu_q^{k-1}(M)$ are Type 1 edges of $\mu_q^k(M)$.  
The following theorem of Stiebitz (whose original appearance in \cite{stieb} may be difficult to locate but which is referenced with proof in \cite{gcones}) will extend our construction of dense triangle-free critical graphs to a nonconstructive proof for the existence of dense critical graphs avoiding all odd cycles beneath any arbitrary threshold:
\begin{theorem}[Stiebitz\cite{stieb}]
\label{t.stieb}  For any odd cycle $C$, we have that $\bar \mu_q^k(C)$ has chromatic number $k+3$.\qed
\end{theorem}
\noindent (The original statement allows the value of $q$ applied in each of the $k$ iterations to vary, but this is unnecessary for us and just serves to complicate notation.) Noting that $\chi(\mu^q(M))=\chi(M)$ and thus also that $\chi(\bar \mu^q(M))\leq \chi(M)+1$, the theorem is equivalent to the following formulation:
\begin{cor}
  For any odd cycle $C$ and any $k+2$ coloring of $\mu_q^k(C)$, all $k+2$ colors appear at the forward vertices.\qed
\label{c.cones}
\end{cor}

With respect to odd cycles, the important aspect of the iterated generalized Mycielskian is the following:

\begin{lemma}
 If $M$ has odd girth $>\ell$ and $q\geq \frac{\ell-1} 2$, then $\bar \mu_q(M)$ has odd girth $>\ell$.  Thus also $\bar\mu_q^k(M)$ has odd girth $>\ell$ for any $k$.
\label{Mogirth}
\end{lemma}
\begin{proof}
  Every vertex in $\mu_q(M)$ is either a vertex $v\in M$, or a corresponding $v_j$ ($1\leq j\leq q$).  If we project all of the vertices $v_j$ onto their corresponding vertices $v$, this induces a map from the edges of $\mu_q(M)$ to the edges of $M$.  This map sends walks to walks and closed walks to closed walks.  Thus if $\mu_q(M)$ contains some odd cycle of length $\leq \ell$, $M$ contains an odd closed walk of length $\leq \ell$, and thus an odd cycle of length $\leq \ell$.  Thus an odd cycle of length $\leq \ell$ in $\bar \mu_q(M)$ must include the extra vertex $z$.

On the other hand if we delete all Type 1 edges from $\bar \mu_q(M)$, the result can be 2-colored, coloring $v_j$ with the parity of $j$ for $0\leq j\leq q$ and $z$ with the parity of $q-1$.  Thus any odd cycle in $\bar \mu_q(M)$ includes an edge of $M$.  Any cycle including an edge of $M$ and the vertex $z$ has length at least $2q+3$.
\end{proof}

For any odd $\ell$ and odd cycle $C_s$ ($s> \ell$), 
Corollary \ref{c.cones} gives that whenever 
$\mu_{q}^{k-2}(C_s)$ is $k$-colored, all $k$ colors must appear at its forward vertices.  Thus for $r\leq k$, and letting $q:=\frac{\ell-3}{2}$, we can now construct a graph $M_{k,r}^{\ell,s}$ from $\mu_{q}^{k-2}(C_s)$ by removing Type 3 edges one by one, discarding any forward vertices isolated by this process, until we cannot do so without allowing the result to be $k$-colored so that only $r-1$ colors appear at forward vertices.  (This is the step where our argument is nonconstructive.)  The result satisfies:
\begin{observ}  $M_{k,r}^{\ell,s}$ is a (connected) graph such that
\label{Ms}
\begin{enumerate}
 \item $M_{k,r}^{\ell,s}$ can be $k$-colored so that just $r$ colors appear at forward vertices.  Moreover, one of those $r$ colors can be required to appear (among forward vertices) at only a single vertex of our choosing.
 \item In any $k$-coloring of $M_{k,r}^{\ell,s}$, at least $r$ colors appear at forward vertices,\label{p.atleastr}
 \item If any edge is removed from $M_{k,r}^{\ell,s}$, the result can be $k$-colored so that at most $r-1$ colors appear at forward vertices.\label{p.atmostr}\qed
\end{enumerate}
\end{observ}
\begin{proof}
  Parts \ref{p.atleastr} and \ref{p.atmostr} are immediate from the definition.  For Part 1, note that by definition, if we remove any Type 3 edge $\{v_q,u_{q-1}\}$ (where $v,u\in \mu^{k-3}_q(C_s)$), the result can be $k$-colored so that only $r-1$ colors appear at forward vertices.  Thus, it suffices to show that we can replace the edge and give a $k$-coloring which is identical at the forward vertices except for using a new color at $v_q$.  This follows from the fact that there is always a vertex (nonadjacent to $v_q$) whose neighborhood is a superset of the neighborhood of $v_q$, whose color can thus be used at $v_q$: if $q=1$ then $v$ is such a vertex; otherwise, $v_{q-2}$ is such a vertex.  Finally, note that the vertex $v_q$ can be chosen arbitrarily among the forward vertices.
\end{proof}

In spite of the fact that these are essentially the same properties enumerated for the sets $U^{k-1}_{k-i+1}$ in Lemma \ref{setcolor} (with the terminology \emph{active vertices} replaced with \emph{forward vertices}), the graphs $M_{k,r}^{\ell,s}$ will not take their place in our constructions in this section.  Instead, we will use the $M_{k,r}^{\ell,s}$ (actually, their forward vertices) in place of the graphs $S_{k-r}$ in the construction of the analogue to the $U^{k-1}_{k-i+1}$ used here.  Let  us now make this precise.

Fixing the parameter $k$, we construct graphs $W^\ell(r_1,r_2,\dots,r_t)$ as follows: take a disjoint union of graphs $M_{k-1,r_i}^{\ell,s}$, $1\leq i\leq t$, together with an `active set' $A=\prod F(M_{k-1,r_i}^{\ell,s})$ (here $F(M)$ denotes the forward vertices of $M$); and join, for each $i$, each vertex $\bar v$ in $F(M_{k-1,r_i}^{\ell,s})$ to all the vertices $u\in A$ which equal $\bar v$ in their $i$th coordinate.  Similar to our convention in Section \ref{construct}, we set

\[_\ell W^{k-1}_{k-i+1}\defn W^\ell(k-1,k-2,\dots,k-i+1).\]
The next lemma is the analogue to Lemma \ref{setcolor} for the sets $_\ell W^{k-1}_{k-i+1}$.
\begin{lemma} The sets $_\ell W^{k-1}_{k-i+1}= W^\ell(k-1,k-2,\dots,k-i+1)$ satisfy:
\label{Wsetcolor}
\begin{enumerate}
\item  \label{Wscolor} $_\ell W^{k-1}_{k-i+1}$ can be properly $(k-1)$-colored so that  $\leq i$ different colors appear at active vertices and so that among active vertices, only one vertex of our choosing gets the $i$th color.
\item  In any $(k-1)$-coloring of  $_\ell W^{k-1}_{k-i+1}$, at least $i$ different colors occur as the colors of active vertices.
\label{Wascolor}
\item If any edge from $_\ell W_{k-i+1}^{k-1}$ is removed, it can be properly $(k-1)$-colored so that at most $i-1$ colors occur at active vertices.
\label{Wascritical}
\end{enumerate}

\end{lemma}
\begin{proof}
  The proof of Lemma \ref{Wsetcolor} is essentially identical to that of Lemma \ref{setcolor}.  It should satisfy the reader to check that the properties of the forward sets of the $M_{k-1,r}$ under $(k-1)$-colorings enumerated in Observation \ref{Ms} are shared by the $r$-critical graphs $S_{k-r}$, and are, moreover, precisely those properties of the $S_{k-r}$ used in the proof of Lemma \ref{setcolor}.
\end{proof}

Similar to our construction in Section \ref{gk}, we take $G^\ell_k$ to be the union of a $C_1= {_\ell W^{k-1}}_{\cel{\frac k 2} +1}$ with a $C_2= {_\ell W^{k-1}}_{\flr{\frac k 2}+1}$; the active sets are joined in a complete bipartite graph.  As in that section, we have (now as consequences of \ref{Wsetcolor}) that $G^\ell_k$ is critically $k$-colorable.  

We now check that $G^\ell_k$ contains no odd cycles of length $\leq \ell$.  As constructed, $G^\ell_k$ was built in part out of many graphs $M_{k-1,r}^{\ell,s}$, and edges of such graphs are of Type 1, 2, or 3; note that the Type 1 edges of an $M_{k,r}^{\ell,s}$ are the edges of a $\mu_{q}^{k-4}(C_s)$.

If we were to delete from $G^\ell_k$ all Type 1 edges of the graphs $M_{k-1,r}^{\ell,s}$ appearing in $G^\ell_k$, we are left with a bipartite graph.   Thus any odd cycle in $G^\ell_k$ must contain a vertex $v$ from one of the graphs $\mu_{q}^{k-4}(C_s)$---without loss of generality we let the vertex $v$ be from $C_1$.  By Lemma \ref{Mogirth}, an odd cycle of length $<\ell$ in $G^\ell_k$ cannot lie entirely in one of the graphs $\mu_{q}^{k-4}(C_s)$, thus it must include at least some active vertex in $C_1$.  Can it contain just one active vertex?  No: if we remove all active vertices but one from $C_1$, the result is some copies of some $M_{k-1,r}^{\ell,s}$'s (where $r$ takes some values $\leq k-1$) which will be disconnected by the removal of the remaining active vertex.  Thus the result contains no cycles passing through the active vertex, and we see that any cycle in $C_1$ passing through one active vertex must pass through at least two.  Of course, since removing the active vertices of $C_1$ from $G^\ell_k$ separates what is left of $C_1$ from $C_2$, it is true more generally that any cycle in $G^\ell_k$ passing through an active vertex of $C_1$ must in fact pass through two such vertices.  The distance between two such vertices is 2, and the distance from each to $v$ is at least $\frac{\ell-3}{2}$, so the length of any odd cycle in $G^\ell_{k}$ is at least $\ell$.

Finally, we remark on the density of the graphs $G^\ell_{k}$.  From a density standpoint, it seems perhaps that there is a problem with our use of the graphs $M_{k-1,r}^{\ell,s}$ in the construction of the graphs $W_{k-i+1}^{k-1}$; we have not given any argument that restricts how small the forward-sets will be of the $M_{k-1,r}^{\ell,s}$ relative to all of the $M_{k-1,r}^{\ell,s}$, so how will we argue that the active sets of the $W^{k-1}_{k-i+1}$ ($i\geq 3$) can take up a $(1-\ep)$ fraction of those graphs by increasing the parameter $s$?  The construction of each $W^{k-1}_{k-i+1}$ involves the use of a $M_{k-1}^{k-1}$, and this graph does not require a deletion argument---we know that half of $M_{k-1}^{k-1}$ consists of forward vertices.   It is not hard to see then, that it is sufficient for our purposes to know that the forward sets of $M_{k-1,k-2}^{\ell,s}$ become arbitrary large as we increase $s$ (even if $\abs{F(M_{k-1,k-2}^{\ell,s})}$ is small compared with $\abs{M_{k-1,k-2}^{\ell,s}}$).  The following simple observation will suffice for our purposes:

\begin{observ}
Given $M_{k,r}^{\ell,s}\sbs \mu_q^{k-3}(C_s)$,
The subset $V_{q-1}=\{v_{q-1}\in M_{k,r}^{\ell,s}\st v\in C_s\}$ is dominated by the subset $V_q=\{v_q\in M_{k,r}^{\ell,s}\st v\in C_s\}=F(M_{k,r}^{\ell,s}).$  In other words, every vertex in $V_{q-1}$ is adjacent to a vertex in $V_q$.
\label{o.dominating}
\end{observ}
\begin{proof}
  Otherwise, $M_{k,r}^{\ell,s}$ can be $k$-colored so that just \emph{one} color appears at the forward vertices.
\end{proof}
Fix now $k$, $r$, and $\ell$, and consider $F(M_{k,r}^{\ell,s})$, and let $U=\{v_{q-1}\in M_{k,r}^{\ell,s}\st v\in C_s\}$.  Note that $\abs{U}=s$, and every vertex of $U$ has degree $2^{k-2}$; thus Observation \ref{o.dominating} implies that $\abs{F(M_{k,r}^{\ell,s})}\geq {s/2^{k-2}}$.  Thus we can, as desired, construct graphs $M_{k-1,k-2}^{\ell,s}$ whose forward sets are arbitrarily large, and thus can construct graphs $W^{k-1}_{k-i+1}$ in which the active set occupies a $(1-\ep)$ fraction of the vertex set.

\bigskip

It is easy to check that our graphs show $c_{\ell,4}\geq \frac 1{(\ell+1)^2}$ (and in this case, the argument is actually constructive as there are no deletions).  With the remarks from the preceding paragraphs out of the way, similar methods as those in Section \ref{density} give that  $c_{\ell,k}=\frac 1 4$ for all $k\geq 6$.  The bound $c_{\ell,5}\geq \frac {1}{2(\ell+1)}$ follows from making each active set equal in size.  Like the case $\ell=3$, the case $\ell=5$ is special in the sense that a better constant $c_{5,5}\geq \frac{3}{35}$ can be shown.  This is achieved by constructing the set $_5W^4_4=W^5(4)$ not with the set $M_{4,4}^{5,s}$ constructed by applying a deletion argument to sparse graph $\mu_1^2(C_s)$, but with a set constructed by applying the deletion argument to the dense graph $\mu_1^1(G_4^5)$, and then carrying out an optimization similar to that in Section \ref{g5} (the best choice is to let the larger of the two active sets be bigger by a factor of $\frac {17}{6}$).

\section{Further Questions}
\label{Qs}

There are many natural questions that follow from what we have presented here.  Probably the most immediate is whether our constants $c_{\ell,k}$ are optimal for the cases $k=4,5$. (Particularly for the case $k=5$, $\ell\geq 7$, improvements may not be out of reach.)  Our current knowledge on the constants $c_{\ell,k}$, coming from the lower bounds proved in this paper and the trivial $\frac 1 4$ upper bound, are summarized in Table \ref{ctable}. (The bound $c_4\geq \frac 1 {16}$ comes from Toft's graph and so is not new).

\renewcommand{\arraystretch}{1.5}

\setlength{\tabcolsep}{4pt}

\begin{table}
\large
\begin{center}
\begin{tabular}{c|cccccc}
\backslashbox[0pt]{$\ell$}{$k$}  & 4      & 5    & 6 & 7 & 8 & $\cdots$ \\
\hline

   3 & {\small $\geq$} $\frac{1}{16}$ & {\small $\geq$} $\frac{4}{31}$    &  $\frac 1 4$  &  $\frac 1 4$ & $\frac 1 4$ & $\cdots$ \\

   5 & {\small $\geq$} $\frac{1}{36}$ & {\small $\geq$} $\frac{3}{35}$    &  $\frac{1}{4}$  &  $\frac 1 4$ &$\frac 1 4$ & $\cdots$ \\

   7 & {\small $\geq$} $\frac{1}{64}$ &  {\small $\geq$} $\frac 1 {16}$  & $\frac 1 4$ &  $\frac 1 4$ &  $\frac 1 4$ &  $\cdots$ \\

   9  & {\small $\geq$} $\frac{1}{100}$ & {\small $\geq$} $\frac 1 {20}$   &  $\frac 1 4$  &$\frac 1 4$  &  $\frac 1 4$ &  $\cdots$ \\
   11  & {\small $\geq$} $\frac{1}{144}$ & {\small $\geq$} $\frac 1 {24}$   &  $\frac 1 4$  &$\frac 1 4$  &  $\frac 1 4$ &  $\cdots$ \\

  $\vdots$ &   $\vdots$  &      $\vdots$  &    $\vdots$    &   $\vdots$ &  $\vdots$ \\
\end{tabular}
\end{center}
\caption{Bounds and exact values for the constants $c_{\ell,k}$.\label{ctable} }
\end{table}

As discussed earlier, the families of triangle-free and odd-girth-$\ell$ graphs seem to be rather unusual in that we have succeeded at determining the maximum asymptotic edge-density in $k$-critical members.  For these families, however, this determination did not require any sophisticated upper-bound result, since it turns out that the maximum edge density is asymptotically unchanged by the $k$-criticality restriction.  In fact, it seems that there is a dearth of upper-bound results about critical graphs---perhaps all known upper bounds come from trivial observations like the fact that a $k$-critical connected graph cannot contain $K_k$ as proper subgraph, which allows an application of Tur\'an's theorem.  (For some recent results and background in the area of the density of general critical graphs, see for example \cite{dc}.)

Because of the added restriction, triangle-free graphs may be a good place to attempt upper bound results on critical graphs.  For example, can we show some nontrivial upper bound on the constant $c_4$?  With its strong requirements on the odd-girth, the following question may be a good testing ground for attempts at upper bounds.

\begin{question}
   What is the behavior of the constants $c_{\ell,4}$ as $\ell$ grows large?  Is it true that $c_{\ell,4}\to 0$ as $\ell\to \infty$?  
\end{question}

One can also return to the question of the size of the minimum degree in highly-chromatic triangle-free graphs.  Since the Erd\H{o}s-Simonovits-Hajnal construction \cite{mindeg} is far from being critical, we ask:
\begin{question}
How large a minimum degree forces the chromatic number of triangle-free (resp.~odd-girth $>\ell$) critical graphs to be bounded?
\end{question}

\noindent Unlike the case where we do not require criticality, we have no linear lower bound.  Thomassen has shown that we cannot have a linear lower bound on the minimum degree if we avoid higher order odd-cycles, even without requiring criticality\cite{pmd}.

\bibliographystyle{abbrv}

\begin{thebibliography}{10}

\bibitem{b}
S.~Brandt.
\newblock On the structure of dense triangle-free graphs.
\newblock {\em Combinatorica}, 20:237--245, 1999.

\bibitem{bt}
S.~Brandt and S.~Thomass{\'e}.
\newblock Dense triangle-free graphs are 4-colorable: A solution to the
  {E}rd{\H o}s-{S}imonovits problem.
\newblock \url{http://www.lirmm.fr/~thomasse/liste/vega11.pdf}.

\bibitem{elZ}
M.~El-Zahar and N.~Sauer.
\newblock The chromatic number of the product of two 4-chromatic graphs is 4.
\newblock {\em Combinatorica}, 5:121--126, 1985.

\bibitem{ekr}
P.~Erd{\H o}s, D.~J. Kleitman, and B.~L. Rothschild.
\newblock Asymptotic enumeration of ${K}\sb{n}$-free graphs.
\newblock In {\em Colloquio Internazionale sulle Teorie Combinatorie (Rome,
  1973), Tomo II}, pages 19--27. Atti dei Convegni Lincei, No. 17. Accad. Naz.
  Lincei, Rome, 1976.

\bibitem{mindeg}
P.~Erd{\H o}s and M.~Simonovits.
\newblock On a valence problem in extremal graph theory.
\newblock {\em Discrete Mathematics}, 5:323--334, 1973.

\bibitem{es}
P.~Erd{\H o}s and A.~H. Stone.
\newblock On the structure of linear graphs.
\newblock {\em Bulletin of the American Mathematical Society}, 52:1087--1091,
  1946.

\bibitem{g}
A.~Gy{\'a}rf{\'a}s.
\newblock Personal communication, 2005.

\bibitem{gcones}
A.~Gy{\'a}rf{\'a}s, T.~Jensen, and M.~Stiebitz.
\newblock On Graphs With Strongly Independent Color-Classes.
\newblock {\em Journal of Graph Theory}, 46:1--14, 2004.

\bibitem{h}
R.~H{\"a}ggkvist.
\newblock Odd cycles of specified length in nonbipartite graphs.
\newblock In {\em Graph Theory}, pages 89--99. Annals of Discrete Mathematics,
  Vol.~13. North-Holland, Amsterdam-New York, 1982.

\bibitem{dc}
T.~R. Jensen.
\newblock Dense critical and vertex-critical graphs.
\newblock {\em Discrete Mathematics}, 258:63--84, 2002.

\bibitem{jin}
G.~Jin.
\newblock Triangle-free four-chromatic graphs.
\newblock {\em Discrete Mathematics}, 145:151--170, 1995.

\bibitem{ks}
A.~V. Kostochka and M.~Stiebnitz.
\newblock On the number of edges in colour-critical graphs and hypergraphs.
\newblock {\em Combinatorica}, 20:521--530, 2000.

\bibitem{k}
I.~K{\v r}{\' i}{\v z}.
\newblock A hypergraph-free construction of highly chromatic graphs without
  short cycles.
\newblock {\em Combinatorica}, 9:227--229, 1989.

\bibitem{lhyper}
L.~Lov{\' a}sz.
\newblock On chromatic number of graphs and set-systems.
\newblock {\em Acta Math. Hungar.}, 19:59--67, 1968.

\bibitem{borsuk}
L.~Lov{\' a}sz.
\newblock Self-dual polytopes and the chromatic number of distance graphs on
  the sphere.
\newblock {\em Acta Scientiarum Mathematicarum}, 45:317--323, 1983.

\bibitem{mycielski}
J.~Mycielski.
\newblock Sur le coloriage des graphes.
\newblock {\em Colloq. Math.}, 3:161--162, 1955.

\bibitem{stableKn}
A.~Schrijver.
\newblock Vertex-critical subgraphs of kneser graphs.
\newblock {\em Nieuw Arch. Wiskd. III. Ser.}, 26:454--461, 1978.

\bibitem{stieb}
M.~Stiebitz.
\newblock {\em Beitr{\" a}ge zur Theorie der f{\" a}rbungskritschen Graphen}.
\newblock PhD thesis, Technical University Ilmenaur, 1985.

\bibitem{cones}
C.~Tardif.
\newblock Fractional chromatic numbers of cones over graphs.
\newblock {\em Journal of Graph Theory}, 38:87--94, 2001.

\bibitem{t}
C.~Thomassen.
\newblock On the chromatic number of triangle-free graphs of large minimum
  degree.
\newblock {\em Combinatorica}, 22:591--596, 2002.

\bibitem{pmd}
C.~Thomassen.
\newblock On the chromatic number of pentagon-free graphs of large minimum
  degree.
\newblock {\em Combinatorica}, 27:241--243, 2007.

\bibitem{toft}
B.~Toft.
\newblock On the maximal number of edges of critical $k$-chromatic graphs.
\newblock {\em Studia Sci.~Math.~Hungar.}, 5:461--470, 1970.

\bibitem{zykov}
A.~Zykov.
\newblock On some properties of linear complexes (in Russian).
\newblock {\em Matem. Sbornik}, 24:163--187, 1949.

\end{thebibliography}

\end{document}